\documentclass{amsart}

\usepackage{amsmath, amsthm, amssymb, amstext, amsfonts}
\usepackage{enumerate}
\usepackage{graphicx}
\usepackage[applemac]{inputenc}

\theoremstyle{plain}

\newtheorem{thm}{Theorem}[section]
\newtheorem{lem}[thm]{Lemma}
\newtheorem{cor}[thm]{Corollary}
\newtheorem{prop}[thm]{Proposition}

\newtheorem*{prob*}{Problem}

\theoremstyle{definition}

\newtheorem{rem}[thm]{Remark}

\newcommand{\cT}{\ensuremath{\mathcal{T}}}

\newcommand{\cV}{\ensuremath{\mathcal{V}}}

\newcommand{\sm}{\ensuremath{\smallsetminus}}

\newcommand{\Aut}{\textnormal{Aut}}

\newcommand{\es}{\ensuremath{\emptyset}}

\newcommand{\sub}{\subseteq}

\def\td{tree-decom\-po\-si\-tion}

\newcommand{\comment}[1]{}

\newcommand{\nat}{{\mathbb N}}
\newcommand{\real}{{\mathbb R}}

\newcommand{\BF}{\ensuremath{\mathcal B}}
\newcommand{\CF}{\ensuremath{\mathcal C}}
\newcommand{\DF}{\ensuremath{\mathcal D}}
\newcommand{\EF}{\ensuremath{\mathcal E}}
\newcommand{\FF}{\ensuremath{\mathcal F}}

\def\?#1{\vadjust{\vbox to 0pt{\vss\vskip-8pt\leftline{%
     \llap{\hbox{\vbox{\pretolerance=-1
     \doublehyphendemerits=0\finalhyphendemerits=0
     \hsize16truemm\tolerance=10000\small
     \lineskip=0pt\lineskiplimit=0pt
     \rightskip=0pt plus16truemm\baselineskip8pt\noindent
     \hskip0pt        
     #1\endgraf}\hskip7truemm}}}\vss}}}

\newenvironment{txteq*}
  {
    \begin{equation*}
    \begin{minipage}[c]{0.85\textwidth} 
    \em                                
  }
  {\end{minipage}\end{equation*}\ignorespacesafterend}

\begin{document}

\title{Accessibility in transitive graphs}
\author{Matthias Hamann}
\address{Matthias Hamann, Department of Mathematics, University of Hamburg, Bundes\-stra\ss e~55, 20146 Hamburg, Germany}
\date{}
\maketitle

\begin{abstract}
We prove that the cut space of any transitive graph $G$ is a finitely generated $\Aut(G)$-module if the same is true for its cycle space.
This confirms a conjecture of Diestel which says that every locally finite transitive graph whose cycle space is generated by cycles of bounded length is accessible.
In addition, it implies Dunwoody's conjecture that locally finite hyperbolic transitive graphs are accessible.
As a further application, we obtain a combinatorial proof of Dunwoody's accessibility theorem of finitely presented groups.
\end{abstract}

\section{Introduction}\label{sec_Intro}

A locally finite transitive graph is \emph{accessible} if there exists some $k\in\nat$ such that any two ends can be separated by at most~$k$ edges.
Relating this notion to the accessibility of finitely generated groups, Thomassen and Woess~\cite{ThomassenWoess} proved that a group is accessible if and only if some (and hence every) of its locally finite Cayley graphs is accessible.

Dunwoody~\cite{D-Access} proved that the finitely presented groups are accessible.
In this paper, we obtain as a corollary of our main theorem a result for the larger class of all locally finite transitive graphs that is similar to the accessibility theorem for finitely generated groups.
Note that we have to make additional assumptions on the graph, as Dunwoody gave examples of locally finite inaccessible transitive graphs, see~\cite{D-AnInaccessibleGroup,D-AnInaccessibleGraph}.\looseness-1

Dunwoody~\cite{D-AnInaccessibleGraph} wrote that it seemed likely that all hyperbolic graphs are accessible.
More generally (see Section~\ref{sec_hyperbolic}), Diestel~\cite{D-personal} conjectured in 2010 that locally finite transitive graphs are accessible as soon as their cycle spaces are generated by cycles of bounded length.
We shall confirm both conjectures and prove the following theorem (for further definitions, see Section~\ref{sec_basics}):

\begin{thm}\label{thm_premain}
Let $G$ be a transitive graph.
If its cycle space is a finitely generated $\Aut(G)$-module, then so is its cut space.
\end{thm}

The following special case of Theorem~\ref{thm_premain} also follows from a result of Cornulier~\cite[Theorem 4.C.3]{Cornulier-Survey} who independently confirmed Diestel's conjecture using simplicial 2-complexes.

\begin{thm}\label{thm_main}
Every locally finite transitive graph whose cycle space is generated by cycles of boun\-ded length is accessible.
\end{thm}

Our proof includes a combinatorial proof of Dunwoody's accessibility theorem~\cite{D-Access} for finitely presented groups (Section~\ref{sec_groups}).
In Section~\ref{sec_hyperbolic}, we will deduce Dunwoody's conjecture on hyperbolic graphs from our main theorem.
Using 2-mani\-folds, Dunwoody~\cite{D-PlanarGraphsAndCovers} proved that locally finite transitive planar graphs are accessible.
In Section~\ref{sec_planar} we will apply Theorem~\ref{thm_main} to give a combinatorial proof for the accessibility of locally finite transitive planar graphs.

Note that we cannot expect an `if and only if' in Theorem~\ref{thm_premain}: Bieri and Strebel~\cite{BieriStrebel} gave an example of a (finitely generated) almost finitely presented group~$G$ that cannot be finitely presented.
This group is accessible by Dunwoody's theorem~\cite{D-Access} so it has an accessible Cayley graph $\Gamma$ and its cut space is finitely generated as $G$-module~\cite[IV.7.6]{DiDu-GroupsGraphs}.%
But as $G$ cannot be finitely presented, the cycle space of~$\Gamma$ cannot be finitely generated as $\Aut(\Gamma)$-module.

\section{Preliminaries}\label{sec_basics}

Let $G$ be a graph.
A \emph{ray} is a one-way infinite path and two rays are \emph{equivalent} if they lie eventually in the same component of $G-F$ for every finite edge set $F\sub E(G)$.
It is easy to show that this is an equivalence relation whose classes are the \emph{edge ends} of~$G$.
An edge set $F\sub E(G)$ \emph{separates} two ends if the rays of these ends lie eventually in different components of $G-F$.

Let $\BF(G)$ be the set of all ordered bipartitions $(A,B)$ such that the set $E(A,B)$ of all edges between $A$ and~$B$ is finite.
So we have $A\cap B=\es$ and $A\cup B=V(G)$.
The set $\BF(G)$ is a vector space over $\mathbb F_2$ where the addition of two ordered bipartitions has as its first component the symmetric difference between the two first components of the summands (and the set of all other vertices as its second component).
For $n\in\nat$ let $\BF_n(G)$ be the subspace induced by the bipartitions $(A,B)$ with \emph{order} at most~$n$, i.e., with $|E(A,B)|\leq n$.
So we have $\BF(G)=\bigcup_{n\in\nat}\BF_n(G)$.
Note that the action of $\Aut(G)$ on~$G$ induces an action of $\Aut(G)$ on $\BF(G)$.
We equip the set $\BF(G)$ with a relation
\[
(A,B)\leq (C,D)\ :\Longleftrightarrow\ A\sub C\text{ and }D\sub B.
\]
It is easy to verify that this is a partial order.
We call a subset $\EF$ of~$\BF(G)$ \emph{nested} if for any two $(A,B),(C,D)\in \EF$ either $(A,B)$ and $(C,D)$ or $(A,B)$ and $(D,C)$ are $\leq$-comparable.
We call $(A,B)\in\BF(G)$ \emph{tight} if the two subgraphs of~$G$ induced by~$A$ and by~$B$ are connected graphs.
The \emph{cut space} of~$G$ is the set of all edge sets $E(A,B)$ with $(A,B)\in\BF(G)$ seen as a vector space over $\mathbb F_2$, where an edge lies in the sum of two \emph{cuts} $E(A,B)$ and $E(A',B')$ if and only if it lies in precisely one of them.
There is a canonical epimorphism between $\BF(G)$ and the cut space of~$G$ that identifies two ordered bipartitions $(A,B)$ and $(C,D)$ if and only if they induce the same cut.

The following theorem is due to Dicks and Dunwoody~\cite{DiDu-GroupsGraphs}.

\begin{thm}{\cite[Theorem 2.20 and Remark 2.21 (ii)]{DiDu-GroupsGraphs}}\label{thm_DD}
If $G$ is a connected graph, then there is a sequence $\EF_1\sub\EF_2\sub\ldots$ of subsets of $\BF(G)$ such that each $\EF_n$ is an $\Aut(G)$-invariant nested set of tight elements of order at most~$n$ that generates~$\BF_n(G)$.\qed
\end{thm}

A useful lemma on tight elements of~$\BF(G)$ is the following -- see e.\,g.~\cite[Proposition 4.1]{ThomassenWoess} for the first part, the second one follows by using the first one for the edges on any fixed path between the two vertices:

\begin{lem}\label{lem_TW4.1}
Let $G$ be a connected graph and $n\in\nat$.
\begin{enumerate}[\rm (1)]
\item \cite[Proposition 4.1]{ThomassenWoess} For every $e\in E(G)$, there are only finitely many tight $(A,B)\in\BF(G)$ with $e\in E(A,B)$ and order at most~$n$.
\item For every $x,y\in V(G)$, there are only finitely many tight $(A,B)\in\BF(G)$ with $x\in A$ and $y\in B$ and order at most~$n$.\qed
\end{enumerate}
\end{lem}

\begin{lem}\label{lem_totalOrder}
Let $G$ be a graph, $x,y\in V(G)$, and $\EF$ be a nested subset of~$\BF(G)$.
Then the partial order $\leq$ is a total order on the set
\[
\EF_{(x,y)}:=\{(A,B)\in \EF\mid x\in A,y\in B\}.
\]
\end{lem}

\begin{proof}
Let $(A,B),(C,D)\in\EF_{(x,y)}$.
If $(A,B)$ and $(D,C)$ are $\leq$-comparable, then either $A\cap C$ or $B\cap D$ is empty both of which is a contradiction to $x\in A\cap C$ and $y\in B\cap D$.
Thus, $(A,B)$ and $(C,D)$ are $\leq$-comparable and the assertion follows immediately.
\end{proof}

The \emph{sum} of finitely many cycles $(C_i)_{i\in I}$ (over $\mathbb F_2$) in~$G$ is the subgraph induced by those edges that occur in an odd number of~$C_i$.
The set of all sums of cycles form a vector space over $\mathbb F_2$, the \emph{cycle space} of~$G$.

\begin{lem}\label{lem_Sequence}
Let $G$ be a graph, $C\sub G$ a cycle, $F$ a finite cut with precisely two edges from~$C$, and $e,f\in F$.
Let $\CF$ be any finite set of cycles in~$G$ such that $C=\sum_{D\in\CF}D$.
Then there is an alternating sequence $e_1C_1e_2C_2\ldots e_n$ of edges $e_i\in F$ and cycles $C_i\in\DF$ with $e=e_1$ and $f=e_n$ such that $e_i$ and $e_{i+1}$ are edges of~$C_i$.
\end{lem}

\begin{proof}
Let $\FF\sub\DF$ consist of precisely those cycles that lie on alternating sequences $e_1C_1\ldots e_n$ with $e=e_1$ and $e_i,e_{i+1}\in E(C_i)$ and $C_i\in\DF$.
Then $\sum_{D\in\FF}|E(D)\cap F|$ is even as every cycle intersects with the finite cut~$F$ in an even number of edges.
Note that each edge in $\bigcup_{D\in\FF}(E(D)\cap F)$ except for $e$ and~$f$ appears an even number of times in~$\DF$ because of $\sum_{D\in\DF}D=C$.
As $e$ lies in $\bigcup_{D\in\FF}(E(D)\cap F)$, we conclude by parity that also $f$ lies in this set.
This proves the assertion.
\end{proof}

The cut space and the cycle space of~$G$ form a natural $\Aut(G)$-module, as $G$ acts canonically on these spaces and this action respect the vector space properties.
We call such an $\Aut(G)$-module \emph{finitely generated} if there are finitely many elements that, together with all its images under $\Aut(G)$, are a generating set for the vector space.

\section{Proof of the main theorem}\label{sec_proof}

In this section, we shall prove our main theorem.
Actually, we will prove the assertion for a slightly larger class of graphs: instead of transitivity we only assume that the graphs are \emph{quasi-transitive}, that is, their automorphism groups have only finitely many orbits on the vertex sets.

\begin{thm}\label{thm_main2}
Let $G$ be a quasi-transitive graph.
If its cycle space is a finitely generated $\Aut(G)$-module, then so is its cut space.
\end{thm}

\begin{proof}
We may assume that $G$ is connected.
Indeed, quasi-transitivity implies that there are only finitely many $\Aut(G)$-orbits on the components of~$G$.
So if we show that the cut space of every component~$C$ is a finitely generated $\Aut(C)$-module, then the cut space of~$G$ is a finitely generated $\Aut(G)$-module.

Note that it suffices to prove that $\BF(G)$ is a finitely generated $\Aut(G)$-module.
Let $\CF$ be a finite generating set of the cycle space of~$G$ as an $\Aut(G)$-module and set
\[
\DF:=\bigcup_{\psi\in\Aut(G)}\CF\psi.
\]
Let~$k$ be the length of a largest cycle in~$\CF$.
Due to Theorem~\ref{thm_DD}, we find for every $n\in\nat$ some $\Aut(G)$-invariant nested $\EF_n$ all of whose elements are tight and such that $\EF_n$ generates $\BF_n(G)$ as a vector space.
Furthermore, we may assume $\EF_n\sub\EF_{n+1}$ for all $n\in\nat$.

For every non-trivial bipartition $\{X,Y\}$ of the vertex set of any cycle $C$, there are only finitely many $(A,B)\in\EF_n$ with $X\sub A$ and $Y\sub B$ and these are totally ordered due to Lemmas~\ref{lem_TW4.1} and~\ref{lem_totalOrder}.
So if there exists some, they have a smallest and a largest element.
Thus, among the elements of~$\EF_n$ that induce a non-trivial (ordered) bipartition on $V(C)$, there are at most $2^{|V(C)|}$ many smallest and at most $2^{|V(C)|}$ many largest such elements.
%
Hence, if $\EF_n$ contains more than $2^{1+k}|\CF|$ orbits, there must be one orbit such that none of its elements is smallest or largest for any bipartition of any $C\in\DF$.
Let $(A,B)$ be an element of such an orbit and let $C\in\DF$ such that $(A,B)$ induces a non-trivial bipartition of~$V(C)$.
As above the set of elements $(A',B')$ of~$\EF_n$ that induce the same bipartition on~$V(C)$ as $(A,B)$, i.\,e.\ $A\cap V(C)=A'\cap V(C)$, forms a finite total order.
So we find a unique largest with the property $(A',B')<(A,B)$ among them.
Note that $E(A,B)$ and $E(A',B')$ coincide on~$C$.

We shall show $A=A'$ and $B=B'$.
Let $C'\in \DF$ such that $C$ and~$C'$ share an edge~$xy$.
Thus, the ordered bipartitions $(A,B)$ and $(A',B')$ induce non-trivial bipartitions on~$V(C')$.
Since any $(E,F)\in\EF_n$ with $(A',B')<(E,F)<(A,B)$ must induce the same bipartition on~$C$ as $(A,B)$, we conclude by the choice of $(A',B')$ that no such element of~$\EF_n$ exists.
By its choice, $(A,B)$ is neither minimal nor maximal with respect to its induced bipartition on~$V(C')$.
Since the bipartitions separating $x$ and~$y$ form a finite total order due to Lemmas~\ref{lem_TW4.1} and~\ref{lem_totalOrder}, the bipartition $(A',B')$ is the unique largest element of~$\EF_n$ that induces the same bipartition of~$V(C')$ as~$(A,B)$.

We have shown
\begin{equation}\tag{$*$}\label{eq_C3KnImpliesHKn3}
\begin{minipage}[t]{0.9\textwidth}
\em
If $C_1,C_2\in\DF$ share an edge, if $(A,B)$ and $(A',B')$ coincide on~$C_1$, and if $(A',B')$ is maximal in~$\EF_n$ among those that are smaller than $(A,B)$ and that coincide with $(A,B)$ on~$C_1$, then $(A,B)$ and $(A',B')$ coincide on~$C_2$.
\end{minipage}
\end{equation}

As $(A,B)$ is tight, we find for any two edges in $E(A,B)$ some cycle that meets $E(A,B)$ in precisely those two edges and the same for $(A',B')$.
Let $e,f\in E(A,B)$ such that $e$ is an edge of~$C$.
So $e$ lies in $E(A',B')$.
We shall show that also $f$ lies in $E(A',B')$.
Due to Lemma~\ref{lem_Sequence}, we find a sequence $e_1C_1e_2\ldots e_n$ with $e_1=e$ and $e_n=f$ such that every $C_i$ lies in~$\DF$.
By adding $eC$ at the beginning of this sequence, we may assume that $C_1=C$.
Applying ($*$), we conclude that $(A,B)$ and $(A',B')$ coincide on~$C_2$.
Inductively, they coincide on every cycle $C_i$, in particular on~$C_n$.
Thus, $f$~-- and hence every edge of $E(A,B)$~-- lies in $E(A',B')$.
So we have $E(A,B)\sub E(A',B')$ and hence $(A,B)=(A',B')$ as $(A,B)$ and $(A',B')$ are tight.
This contradiction shows that there are at most $2^{1+k}|\CF|$ orbits in~$\EF_n$ and that $\BF_n(G)$ is a finitely generated $\Aut(G)$-module.\looseness-1

As every $\EF_n$ has at most $2^{1+k}|\CF|$ orbits, some $\EF_i$ contains maximally many orbits.
Since the orbits of $\EF_i$ are also orbits of~$\EF_j$ for every $j\geq i$, we have $\EF_i=\EF_j$ for all $j\geq i$.
This shows that $\BF(G)$ is a finitely generated $\Aut(G)$-module.
\end{proof}

We extend the notion of accessibility to arbitrary graphs: a graph is \emph{accessible} if there exists some $k\in\nat$ such that any two edge ends can be separated by at most~$k$ edges.

\begin{thm}\label{thm_main3}
Every quasi-transitive graph $G$ whose cycle space is a finitely generated $\Aut(G)$-module is accessible.
\end{thm}

\begin{proof}
By Theorem~\ref{thm_main2}, the space $\BF(G)$ is finitely generated as $\Aut(G)$-module.
If $n$ is the largest size of any cut in some finite generating set, then we obtain $\BF(G)=\BF_n(G)$.

Let $\omega,\omega'$ be two edge ends of~$G$.
Then there is some ordered bipartition $(X,Y)$ such that the edge set $E(X,Y)$ separates $\omega$ and~$\omega'$.
We can write $(X,Y)$ as a sum of finitely many $(A_i,B_i)\in\EF_n$ by the choice of~$\EF_n$.
So some edge set $E(A_i,B_i)$ separates $\omega$ and~$\omega'$, too, by the definition of addition of ordered bipartitions.
\end{proof}

As every locally finite quasi-transitive graph $G$ has only finitely many $\Aut(G)$-orbits of cycles of length at most~$k$, we obtain as a corollary of Theorem~\ref{thm_main3}:

\begin{cor}\label{cor_diestel}
Every locally finite quasi-transitive graph whose cycle space is generated by cycles of bounded length is accessible.\qed
\end{cor}

Interestingly, the assumption of quasi-transitivity in Theorem~\ref{thm_main3} is unnecessary.
In order to show this, we first look at $2$-edge-connected graphs:

\begin{cor}\label{cor_2edgecon}
Let $G$ be a $2$-edge-connected graph whose cycle space is a finitely generated $\Aut(G)$-module.
Then its cut space is a finitely generated $\Aut(G)$-module.
\end{cor}

\begin{proof}
Let $\CF$ be a finite generating set of the cycle space of~$G$ as an $\Aut(G)$-module.
The assumption of $2$-edge-connectivity implies that every edge lies on some cycle and thus on the image of some cycle in~$\CF$ under some automorphism of the graph.
So there are only finitely many orbits on the edges of~$G$ and thus also on the vertices of~$G$.
The assertion follows by Theorem~\ref{thm_main2}.
\end{proof}

Now we use Corollary~\ref{cor_2edgecon} to strengthen Theorem~\ref{thm_main3}:

\begin{thm}\label{thm_main4}
Every graph $G$ whose cycle space is a finitely generated $\Aut(G)$-module is accessible.
\end{thm}

\begin{proof}
Let $\CF$ be a finite generating set of the cycle space of~$G$ as an $\Aut(G)$-module.
As in the proof of Corollary~\ref{cor_2edgecon} we obtain that every maximal $2$-edge-connected subgraph of~$G$ is quasi-transitive and hence, due to Theorem~\ref{thm_main3}, accessible.
Note that there are only finitely many $\Aut(G)$-orbits of maximal $2$-edge-connected subgraphs of~$G$ since every such subgraph must contain some cycle $C\alpha$ with $C\in\CF$ and $\alpha\in\Aut(G)$ and there are only finitely many $\Aut(G)$-orbits of such cycles.
Thus, there is some $n\in\nat$ such that any two edge ends of~$G$ whose rays lie eventually in the same maximal $2$-edge-connected subgraph of~$G$ are separated by at most~$n$ edges.
Rays of different maximal $2$-edge-connected subgraphs are separated eventually by every single edge that separates those two subgraphs.
Furthermore, every ray~$R$ that lies eventually outside every maximal $2$-edge-connected subgraph contains infinitely many edges that are the edges of some cut of size~$1$.
Thus, we can separate $R$ from any other ray it is not equivalent to by some edge of~$R$ eventually.
Hence, $G$ is accessible.
\end{proof}

\section{Applications}

\subsection{Finitely presented groups}\label{sec_groups}

Stallings~\cite{Stallings} proved that every finitely generated group with more than one end splits as a non-trivial free product with amalgamation over a finite subgroup or as an HNN-extension over a finite subgroup.
We can continue this splitting process if one of its factors also has more than one end.
We call a group \emph{accessible} if this splitting stops after finitely many steps.
Wall~\cite{Wall-AccessibilityConjecture} conjectured that every finitely generated group is accessible.
Dunwoody proved that this is true if the group is also finitely presented~\cite{D-Access} but that it is false without the additional assumption~\cite{D-AnInaccessibleGroup}.
Here, we use our result to give a combinatorial proof of Dunwoody's accessibility theorem.

Let $G=\langle S\mid R\rangle$ be a finitely presented group.
Then any word over~$S$ in~$G$ that represents~$1$, is a finite product of relators in~$R$ or its conjugates.
Thus, in its Cayley graph $\Gamma$, every cycle and thus every element of the cycle space is the sum of the elements of the cycle space that are given by elements of~$R$ and its $G$-images.
Hence, we conclude by Theorem~\ref{thm_main} that $\Gamma$ is accessible, which in turn is equivalent to $G$ being accessible due to  Thomassen and Woess~\cite[Theorem 1.1]{ThomassenWoess}.
Note that Diekert and Wei\ss~\cite{DiekertWeiss} offered a combinatorial proof of this equivalence (its original proof applied a result due to Dunwoody~\cite{D-AccessAndGroupCohom} that used some algebraic topology).
So we have just proved combinatorially:

\begin{thm}\cite[Theorem 5.1]{D-Access}\label{thm_D_Access}
Every finitely presented group is accessible.\qed
\end{thm}

\subsection{Hyperbolic graphs}\label{sec_hyperbolic}

For $\delta\in\nat$, we call a graph $G$ \emph{$\delta$-hyperbolic} if it is connected and if for any three vertices and any three shortest paths $P_1,P_2,P_3$, one between any two of the three vertices, every vertex on~$P_1$ has distance at most~$\delta$ to some vertex on~$P_2$ or~$P_3$.
We call $G$ \emph{hyperbolic} if it is $\delta$-hyperbolic for some~$\delta\in\nat$.
A finitely generated group $\Gamma$ is \emph{hyperbolic} if one of its locally finite Cayley graphs is hyperbolic.
As hyperbolic groups are finitely presented (see~\cite{GhHaSur}), they are accessible due to Theorem~\ref{thm_D_Access}.
In this section we will prove the analogue result for quasi-transitive hyperbolic graphs.

It is not hard to show that accessibility is preserved under quasi-isometries.
The same holds for hyperbolicity (see e.\,g.~\cite{GhHaSur}).
But quasi-transitive graphs need not be quasi-isometric to some Cayley graphs due to Eskin et al.~\cite{EFW-DLnotTransitive}.
Thus, we cannot obtain the accessibility of locally finite quasi-transitive hyperbolic graphs directly from the accessibility of finitely generated hyperbolic groups.

\begin{lem}\label{lem_hyperbolic}
Let $G$ be a $\delta$-hyperbolic graph.
Then the cycles of length less than $4\delta+4$ generate its cycles space.
\end{lem}

\begin{proof}
Let us suppose that this is not the case.
Then we take some cycle $C$ that cannot be written as a sum of shorter cycles and whose length is at least $4\delta+4$.
The distance between any two vertices of~$C$ is realized on~$C$ as any shortcut leads to two cycles that are shorter than~$C$ but whose sum is~$C$.
We pick $x,y,z\in V(C)$ such that $d(x,y)=2\delta+2$ and such that $z$ lies in the middle of the longer path between $x$ and~$y$ on~$C$, i.\,e.\ $|d(x,z)-d(y,z)|\leq 1$.
We pick three shortest paths, one between any two of the three vertices, such that their union is~$C$.
So~$z$ does not lie on the chosen shortest path between~$x$ and~$y$.
Let $u$ be the vertex on the shortest path between $x$ and~$y$ that has distance $\delta+1$ to~$x$ and to~$y$.
Then its distance to any vertex on the other two shortest paths is at least $\delta+1$.
This contradiction shows the assertion.
\end{proof}

Dunwoody~\cite{D-AnInaccessibleGraph} thought it likely that every transitive locally finite hyperbolic graph is accessible.
As a direct consequence of Corollary~\ref{cor_diestel} and Lemma~\ref{lem_hyperbolic}, we can confirm this:

\begin{thm}
Every locally finite quasi-transitive hyperbolic graph is accessible.\qed
\end{thm}

\subsection{Planar graphs}\label{sec_planar}

A finitely generated group is \emph{planar} if it has some locally finite planar Cayley graph.
Droms~\cite{D-InfiniteEndedGroups} proved that finitely generated planar groups are accessible.
This is a hint that quasi-transitive planar graphs are also accessible and, indeed, it is true as Dunwoody~\cite{D-PlanarGraphsAndCovers} showed.
As another corollary of Theorem~\ref{thm_main2}, we will prove this combinatorially.

First, we introduce the notion of a degree sequence of orbits because the general idea to prove accessibility for locally finite quasi-transitive planar graphs will mainly be done by induction on this notion.

For a connected locally finite quasi-transitive graph $G$ with $|V(G)|>1$ we call a tupel $(d_1,\ldots, d_m)$ of positive integers with $d_i\geq d_{i+1}$ for all $i<m$ the \emph{degree sequence of the orbits of~$G$} if for some set $\{v_1,\ldots, v_m\}$ of vertices that contains precisely one vertex from each $\Aut(G)$-orbit the degree of~$v_i$ is~$d_i$.
We define an order on the finite tupels of positive integers (and thus on the degree sequences of orbits) by
\[
(d_1,\ldots, d_m)\leq (c_1,\ldots, c_n)
\]
if either $m\leq n$ and $d_i=c_i$ for all $i\leq m$ or $d_i<c_i$ for the smallest $i\leq m$ with $d_i\neq c_i$.
Note that any two finite tupels of positive integers are $\leq$-comparable.

A direct consequence of this definition is the following lemma.

\begin{lem}\label{lem_IndOnDegSequPre}
Any strictly decreasing sequence in the set of finite tupels of positive integers is finite.\qed
\end{lem}

Reformulated for quasi-transitive graphs and their degree sequences of orbits, it reads as follows and enables us to use induction on the degree sequence of the orbits of graphs:

\begin{lem}\label{lem_IndOnDegSequ}
Every sequence of locally finite quasi-transitive graphs whose corresponding sequence of degree sequences of the orbits is strictly decreasing is finite.\qed
\end{lem}

\begin{lem}\label{lem_SmallerDegSequ}
Let $G$ be a locally finite quasi-transitive graph and let $S\sub V(G)$ be such that $G-S$ is disconnected, such that each $S\alpha$ with $\alpha\in\Aut(G)$ meets at most one component of $G-S$, and such that no vertex of~$S$ has all its neighbours in~$S$.
Let $H$ be a maximal subgraph of~$G$ such that no $S\alpha$ with $\alpha\in\Aut(G)$ disconnects~$H$.
Then the degree sequence of the orbits of~$H$ is smaller than the one of~$G$.
\end{lem}

\begin{proof}
First we show that all vertices in~$H$ that lie in a common $\Aut(G)$-orbit of~$G$ and whose degrees in~$G$ and in~$H$ are the same also lie in a common $\Aut(H)$-orbit.
Let $x,y$ be two such vertices and $\alpha\in\Aut(G)$ with $x\alpha=y$.
Suppose that $H\alpha\neq H$.
Then there is some $S\beta$ that separates some vertex of~$H$ from some vertex of~$H\alpha$ by the maximality of~$H$.
But as $y$ and all its neighbours lie in $H$ and in $H\alpha$, they lie in $S\beta$, which is a contradiction to our assumption.
Thus, $\alpha$ fixes~$H$.

We consider vertices $x$ such that $\{x\}\cup N(x)$ lies in no $H\alpha$ with $\alpha\in\Aut(G)$ and such that $x$ has maximum degree with this property.
Let $\{x_1,\ldots, x_m\}$ be a maximal set that contains precisely one vertex from each orbit of those vertices.
If $x_i$ lies outside every $H\alpha$, then no vertex of its orbit is considered for the degree sequence of the orbits of~$H$.
If $x_i$ lies in~$H$, then its degree in some $H\alpha$, say in~$H$, is smaller than its degree in~$G$.
So its value in the degree sequence of orbits of~$H$ is smaller than its value in the degree sequence of orbits of~$G$; but it may be counted multiple times now as the $\Aut(G)$-orbit containing $x_i$ may be splitted into multiple $\Aut(H)$-orbits.
Therefore, we have shown that the degree sequence of orbits of~$H$ is smaller than that of~$G$.
\end{proof}

In order to prove that locally finite quasi-transitive planar graphs are accessible, we reduce the problem of accessibility to the $3$-connected such graphs and then prove the accessibility result for VAP-free planar graphs, where a planar graphs is \emph{VAP-free} if is has an embedding in the plane such that no point of the plane is an accumulation point of the vertex set.
Then we shall prove the result for planar graphs that are not necessarily VAP-free.

Note that the reduction to the $3$-connected case does not use the assumption of planarity.
Thus, it is also applicable in other situations.
The method of this reduction is similar to the method in the proof of Theorem~\ref{thm_main4}, but with edges instead of vertices.
Remember that a \emph{block} of a graph is a maximal $2$-connected subgraph.

For quasi-transitive locally finite graphs, accessibility is equivalent to the existence of some $m\in\nat$ such that any two ends can be separated by removing at most $m$ vertices.
We will use this fact in our proofs without mentioning it there explicitly.\looseness-1

\begin{prop}\label{prop_red1to2}
A locally finite quasi-transitive graph is accessible if and only if each of its blocks is accessible.
\end{prop}

\begin{proof}
Let $G$ be a locally finite quasi-transitive graph.
If it is accessible, then, obviously, each of its blocks must also be accessible.
For the converse, note that every block is uniquely determined by any of its edges.
As there are only finitely many $\Aut(G)$-orbits on the vertices of~$G$, the same holds for the $\Aut(G)$-orbits on the edges.
Thus, there are only finitely many $\Aut(G)$-orbits of blocks.
As each block is accessible, we find some $n\in\nat$ such that any two ends of any common block are separable by at most $n$ vertices.
Remember that the blocks of~$G$ are arranged in a tree-like way (cf.\ the block-cutvertex tree).

Let $R_1,R_2$ be two rays of distinct ends.
If both rays lie eventually in the same block, then we find some separator of at most~$n$ vertices that separates the rays eventually.
If the rays lie eventually in distinct blocks, then they are eventually separated by some $1$-separator that separates these two blocks.
So we may assume that one of the two rays, say $R_1$, does not lie in any block eventually.
Then $R_1$ determines a ray in the block-cutvertex tree \cT\ such that it contains vertices from each block of this ray.
If $R_2$ determines the same ray in~\cT, then $R_1$ and $R_2$ must pass through infinitely many common $1$-separators.
But then, they lie in the same end contrary to their choice.
Thus, $R_2$ does not determine the same ray in~\cT\ as~$R_1$ does and hence, there is some $1$-separator that separates them eventually.
This shows the existence of some $m\in\nat$ such that any two ends are separated by at most~$m$ vertices.
Thus, $G$ is accessible.
\end{proof}

\begin{rem}
In the situation of Proposition~\ref{prop_red1to2} we can take the orbits of the cutvertices one-by-one and apply Lemma~\ref{lem_SmallerDegSequ} for each such orbit.
It follows recursively that each block has a smaller degree sequence of its orbits than the original graph.
\end{rem}

For the reduction to the $3$-connected case for graphs of connectivity~$2$, we apply Tutte's decomposition of $2$-connected graphs into `$3$-connected parts' and cycles.
Tutte~\cite{Tutte} proved it for finite graphs.
Later, it was extended by Droms et al.~\cite{DSS-Tutte} to locally finite graphs.

A \emph{\td} of a graph $G$ is a pair $(T,\cV)$ of a tree $T$ and a family $\cV=(V_t)_{t\in T}$ of vertex sets $V_t\sub V(G)$, one for each vertex of~$T$, such that
\begin{enumerate}[(T1)]
\item $V = \bigcup_{t\in T}V_t$;
\item for every edge $e\in G$ there exists a $t \in V(T)$ such that both ends of $e$ lie in~$V_t$;
\item $V_{t_1} \cap V_{t_3} \sub V_{t_2}$ whenever $t_2$ lies on the $t_1$--$t_3$ path in~$T$.
\end{enumerate}

The sets $V_t$ are the \emph{parts} of $(T,\cV)$ and the intersections $V_{t_1}\cap V_{t_2}$ for edges $t_1t_2$ of~$T$ are its \emph{adhesion sets}; the maximum size of such a set is the \emph{adhesion} of $(T,\cV)$.
Given a part $V_t$, its \emph{torso} is the graph with vertex set $V_t$ and whose edge set is
\[
\{xy\in E(G)\mid x,y\in V_t\}\ \cup\ \{xy\mid \{x,y\}\sub V_t\text{ is an adhesion set}\}.
\]

The automorphisms of~$G$ act canonically on vertex sets of~$G$.
If every part of the \td\ is mapped to another of its parts and this map induces an automorphism of~$T$ then we call the \td\ \emph{$\Aut(G)$-invariant}.

\begin{thm}\cite[Theorem 1]{DSS-Tutte}\label{thm_tutte}
Every locally finite $2$-connected graph $G$ has an $\Aut(G)$-invariant \td\ of adhesion~$2$ each of whose torsos is either $3$-connected or a cycle or a complete graph on two vertices.

In addition, there is a unique such \td\ that has the property that no two torsos corresponding to adjacent vertexs of the \td\ are cycles.
\qed
\end{thm}

We call a \td\ as in the additional statement of Theorem~\ref{thm_tutte} \emph{Tutte's decomposition}.
The idea of the following lemma is basically that of Proposition~\ref{prop_red1to2} but with Tutte's decomposition instead of the block-cutvertex tree.

\begin{prop}\label{prop_red2to3}
A locally finite quasi-transitive $2$-connected graph is accessible if and only if each of its torsos in Tutte's decomposition is accessible.
\end{prop}

\begin{proof}
Let $G$ be a locally finite quasi-transitive $2$-connected graph and $(T,\cV)$ be Tutte's decomposition of~$G$.
Note that every vertex lies in only finitely many $2$-separators (cf.~\cite[Proposition 4.2]{ThomassenWoess}).
Thus, the graph $H$ given by~$G$ together with all edges $xy$, where $\{x,y\}$ forms an adhesion set, is also locally finite and quasi-transitive.
There are only finitely many orbits of (the action induced by) $\Aut(G)$ on~$T$, since any $2$-separator of~$G$ uniquely determines the parts $V_t$ of~$(T,\cV)$ it is contained in and since there are only finitely many $\Aut(G)$-orbits of $2$-separators.\looseness-1

Let $R$ be a ray in~$G$.
If there is some part $V_t$ of~$(T,\cV)$ that is visited infinitely often by~$R$, then consider the following ray in~$H$:
if $R$ leaves~$V_t$ via some $2$-separator $\{x,y\}$, then $R$ must contain both vertices of that adhesion set as it must enter $V_t$ again, which is only possible via the second vertex of $\{x,y\}$.
So if we remove the subpath of~$R$ between $x$ and~$y$ for any such adhesion set $\{x,y\}$, then due to local finiteness we still have a ray in~$H$, which lies in the torso induced by~$V_t$.
So $R$ defines an end of that torso.
Furthermore, if we have a ray in some $3$-connected subgraph $F$ of~$H$, then we reverse the above process to obtain some ray of~$G$.
(We replace the edge $xy$ by some $x$-$y$ path outside of~$F$.)
For two equivalent rays in~$H$, the corresponding rays in~$G$ are still equivalent, so lie in the same end of~$G$.
Thus, we have obtained a canonical bijection between the ends of~$G$ and those of~$H$.
As~$G$ is a subgraph of~$H$, it suffices to prove that $H$ is accessible if and only if each of its maximal $3$-connected subgraphs is accessible.

Obviously, accessiblity of~$H$ implies accessiblity of each maximal $3$-connected subgraph of~$H$.
So let us assume that every maximal $3$-connected subgraph is accessible.
Let $R_1,R_2$ be two non-equivalent rays of~$H$.
By case analysis as in the proof of Proposition~\ref{prop_red1to2}, we obtain the assertion.
\end{proof}

\begin{rem}\label{rem_TutteForInduc}
Unfortunately, we are not able to apply Lemma~\ref{lem_SmallerDegSequ} directly for Proposition~\ref{prop_red2to3} to see that the torsos in Tutte's decomposition have a smaller degree sequence of orbits, as the orbits are not subgraphs of~$G$.
But as not both vertices of any adhesion set have degree~$2$, it is possible to follow the argument of the proof of Lemma~\ref{lem_SmallerDegSequ} for each of the finitely many orbits of the $2$-separators one-by-one to see that each torso has a smaller degree sequence of orbits than~$G$.
\end{rem}

Now we are able to prove the accessibility of quasi-transitive planar graphs in the case of VAP-freeness.

\begin{prop}\label{prop_VAPFreeAccess}
Every locally finite quasi-transitive VAP-free planar graph is accessible.
\end{prop}

\begin{proof}
Due to Propositions~\ref{prop_red1to2} and~\ref{prop_red2to3}, it suffices to show the assertion for $3$-connected graphs.
Every $3$-connected planar graph has (up to homeomorphisms) a unique embedding into the sphere due to Whitney~\cite{whitney_congruent_1932} for finite graphs and Imrich~\cite{I-Whitney} for infinite graphs.
Thus, every automorphism of a $3$-connected infinite locally finite quasi-transitive VAP-free planar graph $G$ induces a homeomorphism of the plane.
So faces are mapped to faces and cycles that are face boundaries are mapped to such cycles.
As $G$ is quasi-transitive, the length of finite face boundaries is bounded by some $n\in\nat$.

Every cycle in~$G$ determines an inner face and an outer face in the plane.
The inner face contains only finitely many vertices as $G$ is VAP-free.
Hence, every cycle is the sum of all face boundaries of the faces that lie in its inner part in the plane and thus, is the sum of cycles of length at most~$n$.
So Corollary~\ref{cor_diestel} implies the assertion.
\end{proof}

Before we tackle the general case, we state a result that can be seen as a (simplified) dual version of Theorem~\ref{thm_DD} for planar graphs.
It follows from a result in~\cite{H-GeneratingCycleSpace}.
In order to state the theorem, we need some definitions.

Let $G$ be a planar graph with planar embedding $\varphi\colon G\to\real^2$.
A \emph{face} of~$\varphi$ is a component of $\real^2\sm \varphi(G)$.
Two cycles $C_1,C_2$ in a planar graph are \emph{nested} if no $C_i$ has vertices in distinct faces of~$\varphi(C_{3-i})$.
A set of cycles is \emph{nested} if every two of its elements are nested.

\begin{thm}\label{thm_nestedCycles}\cite{H-GeneratingCycleSpace}
For every locally finite planar $3$-connected graph~$G$, there is a non-empty $\Aut(G)$-invariant nested set of cycles that generates the cycle space of~$G$.\qed
\end{thm}

We are able to finish the combinatorial proof of Dunwoody's theorem on the accessibility of locally finite quasi-transitive planar graphs.

\begin{thm}\label{thm_PlanarAccess}\cite{D-PlanarGraphsAndCovers}
Every locally finite quasi-transitive planar graph is accessible.
\end{thm}

\begin{proof}
Let $G$ be a locally finite quasi-transitive planar graph.
Due to Propositions~\ref{prop_red1to2} and~\ref{prop_red2to3}, we may assume that $G$ is $3$-connected and due to Proposition~\ref{prop_VAPFreeAccess} we may assume that $G$ is not VAP-free.
Let $\varphi\colon G\to\real^2$ be a planar embedding of~$G$.
Let $\CF$ be a non-empty $\Aut(G)$-invariant nested set of cycles that generates the cycle space of~$G$.
Since $G$ is not VAP-free, there is some cycle $C$ of~$G$ such that both faces of $\real^2\sm \varphi(C)$ contain infinitely many vertices and hence some end of~$G$.
As $\CF$ generates the cycle space as an $\Aut(G)$-module, one of the cycles in~$\CF$ has the same property as~$C$.
Hence, we may assume $C\in\CF$.
In particular, $\{C\alpha\mid\alpha\in\Aut(G)\}$ is nested.\looseness-1

We consider a maximal subgraph $H$ of~$G$ such that no $C\alpha$ with $\alpha\in\Aut(G)$ disconnects~$H$.
In particular, $H$ is connected and for every $C\alpha$ with $\alpha\in\Aut(G)$ one of the faces of ${\real^2\sm \varphi(C\alpha)}$ is empty.
Note that there are only finitely many $\Aut(G)$-orbits of such subgraphs~$H$ as we find in each orbit some element that contains vertices of~$C$ by maximality of~$H$.
So if we show that $H$ is accessible, then the accessibility of~$G$ follows by case analysis as in the proof of~Proposition~\ref{prop_red1to2} since every two distinct such subgraphs $H_1,H_2$ can be separated by some $C\alpha$, so by some finite number of vertices.

Due to Lemma~\ref{lem_SmallerDegSequ}, the graph~$H$ has a strictly smaller degree sequence of its orbits than~$G$ as $C$ disconnects~$G$.
As $H$ is again a locally finite quasi-transitive planar graph, we conclude by induction on the degree sequence of the orbits of such graphs (cf.\ Lemma~\ref{lem_IndOnDegSequ}) with base case if~$G$ is VAP-free that~$H$ is accessible and thus also~$G$.
\end{proof}

We already mentioned that not every locally finite transitive graph is quasi-isometric to some locally finite Cayley graph~\cite{EFW-DLnotTransitive}.
This solved a problem due to Woess~\cite[Problem 1]{Woess-Topo}.
Due to our considerations here, we pose a restrictive version of Woess's problem for planar graphs:

\begin{prob*}\label{prob_2}
Is every locally finite transitive planar graph quasi-isometric to some locally finite planar Cayley graph?
\end{prob*}

\section*{Acknowledgement}

I thank A.~Foremny, K.~Heuer, and T.~Merz for reading earlier versions of this paper and I thank J.~Carmesin for showing me the idea of how to prove Theorem~\ref{thm_PlanarAccess}, once we have Proposition~\ref{prop_VAPFreeAccess}.

\providecommand{\bysame}{\leavevmode\hbox to3em{\hrulefill}\thinspace}
\providecommand{\MR}{\relax\ifhmode\unskip\space\fi MR }
\providecommand{\MRhref}[2]{%
  \href{http://www.ams.org/mathscinet-getitem?mr=#1}{#2}
}
\providecommand{\href}[2]{#2}


\begin{thebibliography}{10}

\bibitem{BieriStrebel}
R.~Bieri and R.~Strebel, \emph{Valuations and finitely presented metabelian
  groups}, Proc.\ London Math.\ Soc. (3) \textbf{41} (1980), 439--464.

\bibitem{Cornulier-Survey}
Y.~Cornulier, \emph{On the quasi-isometric classification of focal hyperbolic
  groups}, arXiv:1212.2229, 2012.

\bibitem{DiDu-GroupsGraphs}
W.~Dicks and M.J. Dunwoody, \emph{Groups {A}cting on {G}raphs}, Cambridge
  Stud.\ Adv.\ Math., vol.~17, Cambridge Univ.\ Press, 1989.

\bibitem{DiekertWeiss}
V.~Diekert and A.~Wei\ss, \emph{Context-free groups and their structure trees},
  Internat.\ J.\ Algebra Comput. \textbf{23} (2013), no.~3, 611--642.

\bibitem{D-personal}
R.~Diestel, \emph{Personal communication}, 2010.

\bibitem{D-InfiniteEndedGroups}
C.~Droms, \emph{Infinite-ended groups with planar {C}ayley graphs}, J.\ Group
  Theory \textbf{9} (2006), 487--496.

\bibitem{DSS-Tutte}
C.~Droms, B.~Servatius, and H.~Servatius, \emph{The structure of locally finite
  two-connected graphs}, Electron.\ J.\ of Comb. \textbf{2} (1995), research
  paper 17.

\bibitem{D-AccessAndGroupCohom}
M.J. Dunwoody, \emph{Accessibility and groups of cohomological dimension one},
  Proc.\ London Math.\ Soc. (3) \textbf{38} (1979), no.~2, 193--215.

\bibitem{D-Access}
\bysame, \emph{The accessibility of finitely presented groups}, Invent.\ Math.
  \textbf{81} (1985), no.~3, 449--457.

\bibitem{D-AnInaccessibleGroup}
\bysame, \emph{An inaccessible group}, Geometric Group Theory (G.A. Niblo and
  M.A. Roller, eds.), L.M.S.\ Lecture Note Ser., vol. 181, Cambridge University
  Press, 1992, pp.~75--78.

\bibitem{D-PlanarGraphsAndCovers}
\bysame, \emph{Planar graphs and covers}, preprint, 2007.

\bibitem{D-AnInaccessibleGraph}
\bysame, \emph{An {I}naccessible {G}raph}, Random walks, boundaries and
  spectra, Progr.\ Probab., vol.~64, Birkh\"auser/Springer Basel AG, 2011,
  pp.~1--14.

\bibitem{EFW-DLnotTransitive}
A.~Eskin, D.~Fisher, and K.~Whyte, \emph{Coarse differentiation of
  quasi-isometries {I}: {S}paces not quasi-isometric to {C}ayley graphs}, Ann.\
  of Math. (2) \textbf{176} (2012), no.~1, 221--260.

\bibitem{GhHaSur}
E.~Ghys and P.~{de la Harpe}, \emph{Sur les groupes hyperboliques, d'apr\`es
  {M}.~{G}romov}, Progress in Math., vol.~83, Birkh\"auser, Boston, 1990.

\bibitem{H-GeneratingCycleSpace}
M.~Hamann, \emph{Generating the cycle space of planar graphs}, submitted,
  arXiv:1411.6392.

\bibitem{I-Whitney}
W.~Imrich, \emph{On {W}hitney's theorem on the unique imbeddability of
  {$3$}-connected planar graphs}, Recent {A}dvances in {G}raph {T}heory
  (M.~Fiedler, ed.), Academia Praha, Prague, 1975, pp.~303--306.

\bibitem{Stallings}
J.R. Stallings, \emph{Group theory and three dimensional manifolds}, Yale
  Math.\ Monographs, vol.~4, Yale Univ.\ Press, New Haven, CN, 1971.

\bibitem{ThomassenWoess}
C.~Thomassen and W.~Woess, \emph{Vertex-transitive graphs and accessibility},
  J.\ Combin.\ Theory (Series B) \textbf{58} (1993), no.~2, 248--268.

\bibitem{Tutte}
W.T. Tutte, \emph{Graph {T}heory}, Cambridge University Press, Cambridge, 1984.

\bibitem{Wall-AccessibilityConjecture}
C.T.C. Wall, \emph{Pairs of relative cohomological dimension one}, J.\ Pure
  Appl.\ Algebra \textbf{1} (1971), 141--154.

\bibitem{whitney_congruent_1932}
H.~Whitney, \emph{Congruent graphs and the connectivity of graphs}, American
  J.\ of Mathematics \textbf{54} (1932), no.~1, 150--168.

\bibitem{Woess-Topo}
W.~Woess, \emph{Topological groups and infinite graphs}, Discrete Math.
  \textbf{95} (1991), no.~1-3, 373--384.

\end{thebibliography}
\end{document}